\documentclass[twoside,reqno]{amsart}
\usepackage{amssymb}
\usepackage[mathscr]{eucal}
\usepackage{hyperref}

\numberwithin{equation}{section}

\theoremstyle{plain}
\newtheorem{prop}{Proposition}[section]
\newtheorem{thm}[prop]{Theorem}
\newtheorem{cor}[prop]{Corollary}
\newtheorem{lem}[prop]{Lemma}
\newtheorem{lemdfn}[prop]{Lemma and Definition}

\theoremstyle{definition}
\newtheorem{dfn}[prop]{Definition}

\newtheorem{lab}[prop]{}

\theoremstyle{remark}

\newtheorem{examples}[prop]{Examples}
\newtheorem{rem}[prop]{Remark}
\newtheorem{rems}[prop]{Remarks}

\newcommand{\isoto}{\overset{\sim}{\to}}
\newcommand{\into}{\hookrightarrow}
\newcommand{\To}{\Rightarrow}
\renewcommand{\iff}{\Leftrightarrow}

\newcommand{\A}{{\mathbb{A}}}

\renewcommand{\P}{{\mathbb{P}}}

\newcommand{\R}{{\mathbb{R}}}
\newcommand{\Z}{{\mathbb{Z}}}

\newcommand{\scrO}{{\mathscr O}}

\DeclareMathOperator{\Pic}{Pic}
\DeclareMathOperator{\Quot}{Quot}
\DeclareMathOperator{\Spec}{Spec}
\DeclareMathOperator{\Sper}{Sper}
\DeclareMathOperator{\supp}{supp}
\DeclareMathOperator{\trdeg}{trdeg}

\DeclareTextFontCommand{\textnf}{\normalfont}

\newcommand{\Cl}{\mathrm{Cl}}
\newcommand{\dom}{\mathrm{dom}}

\newcommand{\inter}{\mathrm{int}}
\newcommand{\reg}{\mathrm{reg}}
\newcommand{\sing}{\mathrm{sing}}

\newcommand{\comp}{\mathbin{\scriptstyle\circ}} 
\renewcommand{\emptyset}{\varnothing}
\newcommand{\ol}{\overline}
\newcommand{\wt}[1]{\widetilde{#1}}
\renewcommand{\setminus}{\smallsetminus}
\renewcommand{\epsilon}{\varepsilon}
\renewcommand{\theta}{\vartheta}

\newcommand{\ratto}{\dashrightarrow}
\newcommand{\spez}{\rightsquigarrow}

\newcommand{\sa}{semi-algebraic}

\begin{document}

\title
[The ring of bounded polynomials on a semi-algebraic set]
{The ring of bounded polynomials on a semi-algebraic set}

\author{Daniel Plaumann}
\address{Universit\"at Konstanz,
      Fachbereich Mathematik und Statistik,
      78457 Konstanz,
      Germany}
\email{daniel.plaumann@uni-konstanz.de, claus.scheiderer@uni-konstanz.de}
\author{Claus Scheiderer}

\date{\today}

\begin{abstract}
Let $V$ be a normal affine $\R$-variety, and let $S$ be a \sa\ subset
of $V(\R)$ which is Zariski dense in $V$. We study the subring $B_V
(S)$ of $\R[V]$ consisting of the polynomials that are bounded on
$S$. We introduce the notion of $S$-compatible completions of $V$,
and we prove the existence of such completions when $\dim(V)\le2$ or
$S=V(\R)$. An $S$-compatible completion $X$ of $V$ yields a ring
isomorphism $\scrO_U(U)\isoto B_V(S)$ for some (concretely specified)
open subvariety $U\supset V$ of $X$. We prove that $B_V(S)$ is a
finitely generated $\R$-algebra if $\dim(V)\le2$ and $S$ is open, and
we show that this result becomes false in general when $\dim(V)\ge3$.
\end{abstract}

\maketitle


\section*{Introduction}

The general question studied in this paper can be stated as follows:
Let $S$ be a \sa\ subset of $\R^n$ (i.e.~a subset described by
polynomial inequalities). How can we describe (conceptually or
explicitly) the ring of polynomials that are bounded on $S$?

To address this question we will work in the following setup. Let
$V$ be an affine variety defined over $\R$, and let $S$ be a \sa\
subset of $V(\R)$.
We write $\R[V]$ for the ring of real polynomial functions on $V$
(the coordinate ring of $V$) and consider its subring
$$B_V(S)=\bigl\{f\in\R[V]\colon f|_S\text{ is bounded}\}.$$
The first systematic study of $B_V(S)$, in the case $S=V(\R)$, was
undertaken in 1996 by Becker and Powers \cite{BP}. Their results have
been generalized substantially
by Monnier \cite{Mo} and Schweighofer \cite{Sw:hol}. The emphasis
there is on iterating the $B_V$ construction (which requires a more
abstract definition via the real spectrum), and on relations to sums
of squares, sums of higher powers, and certificates for positivity;
see also Marshall \cite{Ma}. Motivation came in part from earlier
work on the so-called holomorphy ring in the theory of real rings and
fields (see \cite{Be} and \cite{BP} and references given there). In
particular, rings $A$ were studied in which all elements of the form
$1+\sum_ia_i^2$ with $a_i\in A$ are invertible.
Of course, this condition is hardly ever satisfied for $A=\R[V]$.

A principal difficulty in studying $B_V(S)$ is that these rings need
not be of finite type over $\R$. For example, this is so for
elementary reasons when $S$ is neither relatively compact nor Zariski
dense in $V$ (Cor.\ \ref{thinnotfg}). We will show, however, that
there exist other and more genuine examples as well (Cor.\
\ref{bvsnotft}). In \cite{BP}, \cite{Mo} and \cite{Sw:hol}, this
difficulty was overcome, to a certain extent, by working in the real
spectrum and using sophisticated arguments from real algebra.

Our investigations go in a somewhat different direction. We seek to
understand the structure of $B_V(S)$ in terms of the geometry of $S$,
in particular by constructing algebraic compactifications of $V$ that
are suitably adapted to $S$. To describe the main results of this
paper, let us make the simplifying assumption that the affine variety
$V$ is non-singular and connected. Using resolution of singularities,
we construct a quasi-projective variety $U$ containing $V$ as an open
dense subset such that $\scrO_U(U)$ restricts isomorphically onto
$B_V(V(\R))$ (Thm.\ \ref{gococalxbs} and Thm.\ \ref{exgocovr}). We
also generalize this construction from $S=V(\R)$ to more general \sa\
sets $S\subset V(\R)$. However we have been able to prove a
satisfactory general result only for $\dim(V)\le2$ (Thm.\
\ref{exgocoreginf}).  Combining this with a theorem of Zariski, we
prove that the $\R$-algebra $B_V(S)$ is finitely generated when $\dim
(V)\le2$ and the \sa\ set $S$ satisfies a weak regularity condition
(Thm.\ \ref{bvsfg}).

Here is a brief survey of the contents of this paper. After
introducing the necessary terminology, we fix a connected normal
affine $\R$-variety $V$ and a closed \sa\ set $S\subset V(\R)$, and
we study morphisms $V\to W$ into affine $\R$-varieties $W$ which are
bounded on $S$. From this we obtain a characterization of the field
of fractions of $B_V(S)$ in geometric terms (Prop.\ \ref{omegaprop}).
In particular, we show that $B_V(S)$ has full transcendence degree
$\dim(V)$ over $\R$ if, and only if, there exists a non-constant
$f\in B_V(S)$ such that the set $f^{-1}(c)\cap S$ is compact for
almost all real numbers $c$ (Thm.\ \ref{4conds}).

In Sect.~3 we introduce the notion of compatible completions. Given
a normal affine variety $V$ and a \sa\ set $S\subset V(\R)$, a
complete variety $X$ containing $V$ as an open dense set is said to
be compatible with $S$ if $S$ touches every irreducible component of
$X\setminus V$ in a Zariski dense subset, provided it touches that
component at all. Removing those components from $X$ that are not
touched by $S$, we obtain an open subvariety $U\subset X$ containing
$V$, and we show $\scrO_U(U)\isoto B_V(S)$ (Thm.\ \ref{gococalxbs}).
The existence of compatible completions is studied in Sect.~4.
Relative to $S=V(\R)$, every non-singular affine variety $V$ has a
compatible completion (Thm.\ \ref{exgocovr}). Relative to more
general \sa\ subsets of $V(\R)$, we can prove such a result for $\dim
(V)\le2$ (Thm.\ \ref{exgocoreginf}). Turning things around, we prove
in Thm.\
\ref{owalsbs} that for every normal real quasi-projective variety $U$
there is an open affine subset $V\subset U$ and a \sa\ set $S\subset
V(\R)$ with $\scrO_U(U)\isoto B_V(S)$. In Sect.~5 we finally study
finite generation of the $\R$-algebra $B_V(S)$. After showing that
$B_V(S)$ is not noetherian when $S$ is not Zariski dense in $V$,
we prove finite generation of $B_V(S)$ for $\dim(V)\le2$ (Thm.\
\ref{bvsfg}), using Zariski's theorem. We also show that this result
becomes false when $\dim(V)\ge3$.

\smallskip \emph{Acknowledgements:} Concerning the final section of
this paper, we greatly profitted from a correspondence with Sebastian
Krug, who provided us with examples and constructions in the context
of Hilbert's fourteenth problem, going beyond what went into this
paper.  We are also grateful to Fabrizio Catanese for suggesting the
connection to Zariski's work.


\section{Notations and conventions}

\begin{lab}\label{rvar}
By an $\R$-variety we mean a separated and reduced $\R$-scheme $V$ of
finite type, not necessarily irreducible. The structural sheaf of $V$
is written $\scrO_V$. If $V$ is affine, we write $\R[V]:=\scrO_V(V)$
for the coordinate ring of $V$. If $V$ is irreducible, then $\R(V)$
denotes the function field of $V$.

The set $V(\R)$ of $\R$-rational points is endowed with the euclidean
topology. The notion of \sa\ subsets of $V(\R)$ is well-known when
$V$ is affine, and is easily transferred to the general case, a
subset $S\subset V(\R)$ being \sa\ if and only if $S\cap U(\R)$ is
\sa\ in $U(\R)$ for every open affine subset $U$ of $V$.

An irreducible $\R$-variety $V$ is said to be \emph{real} if it has a
non-singular $\R$-point, or equivalently, if the function field
$\R(V)$ is (formally) real. It is also equivalent that $V(\R)$ is
Zariski dense in $V$.
\end{lab}

\begin{lab}
Let $f\colon V\to W$ be a morphism of $\R$-varieties, and let $S$ be
a subset of $V(\R)$. We say that $f$ is \emph{bounded on $S$} if the
closure of $f(S)$ in $W(\R)$ is compact.

In particular, this applies in the case $W=\A^1$, i.e., when $f$ is
a regular function on $V$. The main object of this paper is to study
the subring
$$B_V(S):=\{f\in\R[V]\colon f\text{ is bounded on }S\}$$
of $\R[V]$, for $V$ an affine $\R$-variety. Here are some immediate
observations:
\end{lab}

\begin{lem}\label{soriten}
Let $V$ be an affine $\R$-variety, let $S$, $S'$ be subsets of
$V(\R)$.
\begin{itemize}
\item[(a)]
$B_V(\ol S)=B_V(S)$;
\item[(b)]
$B_V(S)=\R[V]$ if and only if $\ol S$ is compact;
\item[(c)]
$B_V(S\cup S')=B_V(S)\cap B_V(S')$;
\item[(d)]
if $W$ is the (reduced) Zariski closure of $S$ in $V$ and $I$ is the
vanishing ideal of $W$ in $\R[V]$, then $B_V(S)$ contains $I$,
and $B_W(S)=B_V(S)/I$ as subrings of $\R[V]/I=\R[W]$;
\item[(e)]
the subring $B_V(S)$ is relatively integrally closed in $\R[V]$.
\end{itemize}
\end{lem}

\begin{proof}
For (e), observe that a relation $f^n+a_1f^{n-1}+\cdots+a_n=0$ with
$f\in\R[V]$ and $a_1,\dots,a_n\in B_V(S)$ implies that $|f|\le1+\max
_i|a_i|$ holds (pointwise) on $V(\R)$, hence $f$ is bounded on $S$.
\end{proof}

\begin{lab}
For the entire paper, our base field will be the field $\R$ of real
numbers. All results remain true when $\R$ is replaced by an
arbitrary real closed field and bounded is replaced by bounded over
$R$, compact by \sa ally compact etc.
\end{lab}


\section{Fibres of bounded morphisms}

Throughout this section we assume that $V$ is an \emph{irreducible}
affine $\R$-variety and that $S$ is a fixed \sa\ subset of $V(\R)$.

\begin{lab}
Let $\varphi\colon V\to W$ be a morphism of $\R$-varieties. Given
$y\in W(\R)$, we write $S_y:=\{x\in S\colon\varphi(x)=y\}$ for the
fibre of $y$ in $S$. Given $f\in\R[V]$, we define
$$\Omega_\varphi(f)\>:=\>\bigl\{y\in W(\R)\colon f\text{ is unbounded
on }S_y\bigr\}.$$
This is a \sa\ subset of $W(\R)$ which is contained in $\varphi(S)$.
\end{lab}

\begin{prop}\label{omegaprop}
Let $B=B_V(S)$. For $f\in\R[V]$, the following are equivalent:
\begin{itemize}
\item[(i)]
$f\in\Quot(B)$;
\item[(ii)]
there exists a dominant morphism $\varphi\colon V\to W$ of affine
$\R$-varieties which is bounded on $S$ such that $\Omega_\varphi(f)$
is not Zariski dense in $W$.
\end{itemize}
\end{prop}

\begin{proof}
When $V$ is an affine $\R$-variety and $f\in\R[V]$, we write $Z(f):=
\{x\in V(\R)$: $f(x)=0\}$ for the zero set of $f$ in $V(\R)$. More
generally, if $M\subset V(\R)$ is a \sa\ subset and $f\colon M\to\R$
is a (\sa) map, we denote the zero set of $f$ in $M$ by $Z(f)$.

Assume $f\in\Quot(B)$, so there is $0\ne h\in B$ with $fh\in B$. Let
$C$ be any finitely generated subalgebra of $B$ containing $h$ and
$fh$, write $W=\Spec(C)$, and let $\varphi\colon V\to W$ be the
morphism induced by the inclusion $C\subset\R[V]$. Then $\varphi$ is
bounded on $S$. Let $y\in\Omega_\varphi(f)$, so $f$ is unbounded on
$S_y=S\cap\varphi^{-1}(y)$. On the other hand, $h$ and $fh$ lie in
$C=\R[W]$, so they are constant on $S_y$. Together, this implies
$h(y)=0$. So $\Omega_\varphi(f)$ is contained in $Z(h)$, and is
therefore not Zariski dense in $W$.

Conversely assume that there is a dominant morphism $\varphi\colon
V\to W$ as in (ii). Via $\varphi$, we consider $\R[W]$ as a subring
of $B_V(S)\subset\R[V]$. Write $\Omega:=\Omega_\varphi(f)$. We first
define a map $\tilde f\colon\ol{\varphi(S)}\to\R$ by
$$\tilde f(y)\>:=\>\begin{cases}\Bigl(1+\sup\{|f(x)|\colon x\in S_y\}
\Bigr)^{-1}&\text{if }S_y\ne\emptyset,\\\hfill1\hfill&\text{if }S_y=
\emptyset,\end{cases}$$
($y\in\ol{\varphi(S)}$), where we put $\frac1\infty:=0$. The map 
$\tilde f$ has \sa\ graph, and $\tilde f^{-1}(0)=\Omega$. Let $D$ be 
the set of points in $\ol{\varphi(S)}$ where $\tilde f$ fails to be 
continuous. Then $D$ is not Zariski dense in $W$. This can, for 
instance, easily be deduced from a cylindrical algebraic 
decomposition of the graph of $\tilde f$ (see \cite{BCR} Thm.\ 
2.3.1). Hence $D\cup\Omega$ is not Zariski dense in $W$ either.

By Lemma \ref{contminorant} below, there exists a continuous map
$\tilde h\colon\ol{\varphi(S)}\to R$ with \sa\ graph such that
$|\tilde h|\le\tilde f$ on $\ol{\varphi(S)}$, and such that
$Z(\tilde h)\subset\ol D\cup\Omega$. From the definition of
$\tilde f$ it is clear that $|f(x)\,\tilde h(\varphi(x))|\le|f(x)\,
\tilde f(\varphi(x))|<1$ for every $x\in S$.

Since $Z(\tilde h)$ is not Zariski dense in $W$, there exists $h\ne0$
in $\R[W]$ with $Z(\tilde h)\subset Z(h)$. Since $\ol{\varphi(S)}$ is
compact, the \L ojasiewicz inequality (see \cite{BCR}, Cor.~2.6.7)
implies that there
are $0<c\in\R$ and a positive integer $N$ such that $|h^N|\le c\cdot
|\tilde h|$ on $\ol{\varphi(S)}$. We conclude
$$|f(x)\,h(\varphi(x))^N|\le c\cdot|f(x)\,\tilde h(\varphi(x))|<c$$
for all $x\in S$. This shows that $fh^n$ is bounded on $S$, and so
$f=(fh^n)/h^n$ lies in $\Quot(B)$.
\end{proof}

The following easy fact was used in the last proof:

\begin{lem}\label{contminorant}
Let $M\subset\R^n$ be a \sa\ subset, and let $f\colon M\to\R$ be a
map with \sa\ graph and with $f\ge0$ on $M$. Let $D$ be the set of
points in $M$ where $f$ fails to be continuous. Then there exists a
continuous \sa\ function $g\colon M\to\R$ such that $|g|\le f$ on $M$
and such that $Z(g)\subset Z(f)\cup\ol D$.
\end{lem}

\begin{proof}
Upon replacing $f$ by $\min\{f,1\}$ we may assume $f\le1$ on $M$.
The distance function $d_D\colon\R^n\to\R$, $d_D(x)=\inf\{|x-y|\colon
y\in D\}$ is continuous with \sa\ graph, and it vanishes precisely on
$\ol D$. The function
$$g(x)\>:=\>f(x)\cdot\frac{d_D(x)}{1+d_D(x)}\quad(x\in M)$$
has the desired properties.
Indeed, $f$ is continuous outside of
$\ol D$, and hence so is $g$. Fix $x\in M\cap\ol D$. Then $g(x)=0$,
and for all $y\in M$ we have $d_D(y)\le|y-x|$. This implies
$$0\le g(y)\le\frac{|y-x|}{1+|y-x|}$$
for all $y\in M$.
In particular, $g$ is continuous in $x$.
\end{proof}

\begin{lab}
We fix an affine irreducible $\R$-variety $V$ and a \sa\ set
$S\subset V(\R)$. Given a morphism $\varphi\colon V\to W$ of
$\R$-varieties, we put
$$\Omega_\varphi:=\bigcup_{f\in\R[V]}\Omega_\varphi(f).$$
If $\R[V]$ is generated by $x_1,\dots,x_n$ as an $\R$-algebra then
$\Omega_\varphi=\bigcup_{i=1}^n\Omega_\varphi(x_i)$. This shows that
$\Omega_\varphi$ is a \sa\ subset of $W(\R)$. Clearly,
$$\Omega_\varphi\>=\>\bigl\{y\in W(\R)\colon\ol{S_y}\text{ is not
compact}\bigr\}.$$
\end{lab}

Given an $\R$-algebra $B$, we define the transcendence degree $\trdeg
_\R(B)$ of $B$ as the maximum number of elements of $B$ which are
algebraically independent over~$\R$.

\begin{thm}\label{4conds}
Let $S\subset V(\R)$ be a \sa\ set, let $B=B_V(S)$. The
following conditions are equivalent:
\begin{itemize}
\item[(i)]
$\Quot(B)=\R(V)$;
\item[(ii)]
$\trdeg_\R(B)=\dim(V)$;
\item[(iii)]
there is a birational morphism $V\to W$ of affine varieties which is
bounded on $S$;
\item[(iv)]
there is a dominant morphism $\varphi\colon V\to W$ of affine
varieties which is bounded on $S$ such that $\Omega_\varphi$ is not
Zariski dense in $W$;
\item[(v)]
there is a non-constant $f\in\R[V]$, bounded on $S$, such that the
set $f^{-1}(c)\cap\ol S$ is compact for every $0\ne c\in\R$.
\end{itemize}
\end{thm}

Note that condition (v) implies that there exists $\varphi\colon V\to
W$ as in (iv) with $W=\A^1$.

\begin{proof}
(i) $\To$ (iii):
Assume $\Quot(B)=\R(V)$. Choose a finitely generated $\R$-subalgebra
$C$ of $B$ with $\Quot(C)=\R(V)$, and put $W=\Spec(C)$. The morphism
$\varphi\colon V\to W$ induced by the inclusion $C\subset\R[V]$
satisfies condition (iii).

(iii) $\To$ (ii):
(iii) implies $\R[W]\into B$, which implies (ii).

(ii) $\To$ (i):
The field extension $\Quot(B)\subset\R(V)$ is algebraic, by (ii).
Therefore, given $f\in\R[V]$, there exists $0\ne t\in B$ such that
$tf$ is integral over $B$. Since $B$ is integrally closed in $\R[V]$
(Lemma \ref{soriten}(e)), this shows $tf\in B$, and hence $f\in\Quot
(B)$.

(iii) $\To$ (v):
For $\varphi\colon V\to W$ as in (iii), let $D$ be a proper closed
subvariety of $W$ such that the restriction $\varphi^{-1}
(W\setminus D)\to W\setminus D$ of $\varphi$ is an isomorphism.
Choose a non-constant function $g\in\R[W]$ which vanishes on $D$ and
on the boundary of $\varphi(S)$ in $W(\R)$. Let $0\ne c\in\R$. Then
$g^{-1}(c)\cap\varphi(S)=g^{-1}(c)\cap\ol{\varphi(S)}$ by the choice
of $g$, and this set is compact since $\ol{\varphi(S)}$ is compact.
Moreover, $g^{-1}(c)\cap\varphi(S)$ is contained in $(W\setminus D)
(\R)$, and so the preimage
$$\varphi^{-1}\bigl(g^{-1}(c)\cap\varphi(S)\bigr)\>=\>(g\comp\varphi)
^{-1}(c)\cap\varphi^{-1}(\varphi(S))$$
is compact as well. In particular, $(g\comp\varphi)^{-1}(c)\cap\ol S$
is compact, and so it suffices to take $f:=g\comp\varphi$.

(v) $\To$ (iv)
is trivial (we can take $W=\A^1$ and $\varphi:=f$).

(iv) $\To$ (i):
Let $\varphi$ be as in (iv). Given any $f\in\R[V]$, the set $\Omega_
\varphi(f)$ is not Zariski dense in $W$. So Prop.\ \ref{omegaprop}
shows $f\in\Quot(B)$.
\end{proof}

Under suitable conditions, we know a~priori that the $\R$-algebra $B$
is finitely generated (see Thm.\ \ref{bvsfg} below). In these cases
we can formulate Proposition \ref{omegaprop} and Theorem \ref{4conds}
more succinctly:

\begin{cor}
Assume that the $\R$-algebra $B=B_V(S)$ is finitely generated, put
$W=\Spec(B)$, and let $\varphi\colon V\to W$ be the canonical
morphism.
\begin{itemize}
\item[(a)]
$f\in\R[V]$ lies in $\Quot(B)$ if and only if $\Omega_\varphi(f)$ is
not Zariski dense in $W$.
\item[(b)]
$\varphi$ is birational if and only if $\dim(W)=\dim(V)$, if and only
if the set $\Omega_\varphi$ is not Zariski dense in $W(\R)$.
\qed
\end{itemize}
\end{cor}

\begin{rem}
Even if $B=B_V(S)$ fails to be finitely generated, we can
characterize $\trdeg_\R(B)$ as the maximum dimension of an affine
$\R$-variety $W$ for which there exists a dominant morphism $V\to W$
which is bounded on $S$.
\end{rem}

\begin{examples}
\hfil

1.\
Let $V=\A^2$, and consider the set $S=\{(x,y)\colon0\le x(x^2+y^2)\le
1\}$ in $\R^2$. Then $y\notin B(S)$ but $x$, $xy\in B(S)$, so $y\in
\Quot(B)$. In fact, $B=\R[x,xy,xy^2]$.
If $\varphi\colon V\to W=\Spec(B)$ denotes the canonical map then
$\Omega_\varphi=\Omega_\varphi(y)$ consists only of the origin in
$W$.

2.\
Let $f$, $g\in\R[x,y]$ be two algebraically independent polynomials,
and let $S=\{p\in\R^2\colon|f(p)|\le1$, $|f(p)g(p)|\le1\}$. Clearly
$f$, $fg\in B(S)$, hence $\trdeg B(S)=2$. From Theorem \ref{4conds}
it follows that there exists a non-constant polynomial $h\in B(S)$
such that all fibres $h^{-1}(c)\cap S$ are compact for $c\ne0$.
Depending on the choice of $f$ and $g$, it may not be \emph{a priori}
clear how to find such $h$.
For a concrete example, take $f=x^2y$ and $g=y^2$. Here, neither $f$
nor $fg$ have any compact fibres, but (for example) $h=xy$ will do.

3.\
The same example also shows another phenomenon. If $\trdeg B(S)=2$,
there exists a birational morphism $\varphi\colon\A^2\to W$ of affine
varieties that is bounded on $S$. In many instances, the map $\A^2\to
\A^2$, $(x,y)\mapsto((f(x),f(y)g(y))$ will not have this property,
implying in particular that $B(S)$ is strictly larger than
$\R[f,fg]$. Finding a birational morphism $\varphi$ as above can be
seen as a first step towards determining $B(S)$. (For $f$ and $g$ as
in the previous example, it is easy to see that $B(S)=\R[xy,x^2y,
x^2y^3]$, while the map $(x,y)\mapsto(f(x),f(y)g(y))$ has generically
degree four.) There does not seem to be any general procedure that
will produce such $\varphi$. See however Remark \ref{excompl}.3.
\end{examples}

\begin{lab}
Throughout this paper, we shall assume for the most part that
varieties are irreducible. Here are a few remarks on the reducible
case. Let $V$ be an affine $\R$-variety with irreducible components
$V_1,\dots,V_r$, let $S$ be a \sa\ subset of $V(\R)$, and write $S_i
:=S\cap V_i(\R)$. The relation between $B_V(S)$ and the rings
$B_{V_i}(S_i)$ ($i=1,\dots,r$) depends largely on the way the
components $V_i$ of $V$ meet. Clearly, the restriction $\R[V]\to
\R[V_i]$ maps $B_V(S)$ to $B_{V_i}(S_i)$ ($i=1,\dots,r$). But $B_V(S)
\to B_{V_i}(S_i)$ need not be surjective, as the following example
shows. Let $V$ be the plane affine curve $x(x^2+y^2-1)=0$, which is
the union of a circle $V_1$ and a line $V_2$, and take $S=V(\R)$.
Then $B_V(S)$ consists of those polynomials which are constant on the
line, and so the restriction map $B_V(S)\to B_{V_1}(S_1)$ to the
circle does not contain $y|_{V_1}$ in its image.

Nevertheless, in the case when $V$ is a curve, the relation between
$B_V(S)$ and the $B_{V_i}(S_i)$ is understood farily well, see
\cite{Pl} for details.
Things become considerably more difficult in dimension at least two.
For instance, while $B_V(S)=\R$ implies $B_{V_i}(S_i)=\R$ for all $i$
when $\dim(V)=1$ (\cite{Pl}, Lemma 1.12\,(5)), this is false in
higher dimensions. For an example, let $V$ be the union of two planes
in $3$-space,
and let $S$ be the union of the first plane $V_1(\R)$ with a strip
$[0,1]\times\R$ in the second plane, where the strip is transversal
to the line $V_1\cap V_2$.
Then $B_V(S)=\R\ne B_{V_2}(S_2)$.
\end{lab}


\section{Bounded polynomials and completions of varieties}
\label{goco}

\begin{lab}
Let $X$ be a normal $\R$-variety. A \emph{prime divisor} on $X$ will
be a closed irreducible subset of codimension one in $X$. By a
\emph{divisor} on $X$ we always mean a Weil divisor, that is, an
element of the free abelian group generated by the prime divisors on
$X$. Linear equivalence of divisors is denoted~$\sim$. If $Z$ is a
prime divisor, the discrete valuation of $\R(X)$ associated with $Z$
will be denoted by $v_Z$. Thus $v_Z(f)$ is the vanishing order of $f$
along $Z$, for $f\in\R(X)^*$. Recall that the prime divisor $Z$ is
said to be real if its residue field $\R(Z)$ can be ordered
(c.f.~\ref{rvar}).

Given a rational map $f\colon V\ratto W$ between irreducible
varieties, we denote by $\dom(f)$ the largest open subset of $V$ on
which $f$ is defined.
\end{lab}

\begin{lab}
Let $X$ be an $\R$-variety. Sometimes it will be convenient to work
in the real spectrum $X_r$ of $X$. When $X$ is affine, $X_r=\Sper
\R[X]$ is the space of all orderings of the ring $\R[X]$ (see
\cite{BCR} \S\,7). When $X$ is not necessarily affine, the
topological space $X_r$ is defined by glueing the real spectra
$(U_i)_r$ ($i=1,\dots,r$) of an open affine cover $X=U_1\cup\cdots
\cup U_r$, see \cite{sch:sln} 0.4 for more details. Thus a point
$\alpha\in X_r$ corresponds to a pair $(x_\alpha,P_\alpha)$ where
$x_\alpha$ is a (scheme-theoretic) point of $X$ and $P_\alpha$ is an
ordering of the residue class field of $X$ in $x_\alpha$. The support
of $\alpha$ is $\supp(\alpha)=\ol{\{x_\alpha\}}$. In particular,
$X(\R)$ is a topological subspace of $X_r$ in the obvious way.
Similarly, when $X$ is irreducible, the space $\Sper\R(X)$ of all
orderings of the function field of $X$ is identified with the
subspace $\{\alpha\in X_r\colon\supp(\alpha)=X\}$ of~$X_r$.

The topological space $X_r$ is spectral, and hence there is a
well-defined notion of constructible subsets of $X_r$ (see
\cite{sch:sln} 0.4). For every \sa\ subset $S$ of $X(\R)$, there
exists a unique constructible subset $\wt S$ of $X_r$ such that $S=
\wt S\cap X(\R)$. If $X$ is affine, then $\wt S$ is the subset of
$X_r$ defined by the same system of inequalities as $S$. It is
well-known that $\wt S$ is open (resp.\ closed) in $X_r$ if and only
if the same is true of $S$ in $X(\R)$. For two points $\alpha$,
$\beta\in X_r$, we say that $\alpha$ specializes to $\beta$, denoted
$\alpha\spez\beta$, if $\beta$ is contained in the closure of
$\alpha$.
\end{lab}

Recall that a valuation $v$ of a field $K$ is called
\emph{compatible} with an ordering $\le$ of $K$ if $0<b\le a$, for
$a$, $b\in K^*$, implies $v(b)\ge v(a)$. The usefulness of the real
spectrum in our context comes from the following lemma:

\begin{lemdfn}\label{dfncomp}
Let $X$ be a normal $\R$-variety, let $Z$ be a prime divisor on $X$,
and let $S$ be a \sa\ subset of $X(\R)$. The following conditions are
equivalent:
\begin{itemize}
\item[(i)]
$Z(\R)\cap\ol{(S\cap U(\R))}$ is Zariski dense in $Z$, where $U=X
\setminus Z$;
\item[(ii)]
there is a specialization $\alpha\spez\beta$ in $X_r$ with $\alpha\in
\wt S$, $\supp(\beta)=Z$ and $\alpha\ne\beta$;
\item[(iii)]
the discrete valuation $v_Z$ of $\R(X)$ is compatible with some
ordering in $\wt S\cap\Sper\R(X)$.
\end{itemize}
If these conditions hold we say that $Z$ and $S$ are
\emph{compatible}.
\end{lemdfn}

Note that every prime divisor which is compatible with $S$ has a real
residue field.

\begin{proof}
Let $T:=Z(\R)\cap\ol{(S\cap U(\R))}$, a closed \sa\ subset of
$Z(\R)$. The set $T$ is Zariski dense in $Z$ if and only if $\wt T$
contains a point with support $Z$. The latter condition is equivalent
to (ii), and (iii) is merely a reformulation of (ii).
\end{proof}

The following lemma is obvious (see \cite{Sch:TAMS} Lemma 0.2) and
will be used frequently:

\begin{lem}
Let $V$ be a connected normal $\R$-variety, let $S\subset V(\R)$ be a
\sa\ set, and let $Z$ be a prime divisor in $V$ which is compatible
with $S$. Let $f_1,\dots,f_r\in\R(V)$ satisfy $f_i\ge0$ on $S\cap\dom
(f)$. Then
$$v_Z\Bigl(\sum_if_i\Bigr)\>=\>\min_iv_Z(f_i).\eqno\square$$
\end{lem}

\begin{lem}\label{bddfctcompat}
Let $V$ be a normal $\R$-variety, let $S\subset V(\R)$ be a \sa\ set,
and let $Z$ be a prime divisor in $V$ which is compatible with $S$.
If a rational function $f\in\R(V)^*$ has a pole along $Z$, then $f$
is unbounded on $S\cap\dom(f)$.
\end{lem}

\begin{proof}
The compatibility of $Z$ with $S$ means that there exists $\alpha\in
\wt S\cap\Sper\R(V)$ which makes the discrete valuation ring $\scrO_
{V,Z}$ convex in $\R(V)$. Assume that $f$ is bounded on $S':=S\cap
\dom(f)$. Since $\wt S'$ has the same trace in $\Sper\R(V)$ as
$\wt S$, we have $\alpha\in\wt S'$ as well. On the other hand, $f$
bounded on $S'$ implies that $f$ lies in the $\alpha$-convex hull of
$\R$ in $\R(V)$. In particular, $f\in\scrO_{V,Z}$, which means that
$f$ does not have a pole along $Z$.
\end{proof}

\begin{dfn}\label{dfngoodcompl}
Let $V$ be an irreducible $\R$-variety and let $S$ be a \sa\ subset
of $V(\R)$. By a \emph{completion} of $V$ we mean an open dense
embedding $V\into X$ into a complete $\R$-variety. The completion $X$
will be said to be \emph{compatible with $S$} (or
\emph{$S$-compatible}) if for every irreducible component $Z$ of
$X\setminus V$ the following conditions hold:
\begin{itemize}
\item[(1)]
The local ring $\scrO_{X,Z}$ is a discrete valuation ring;
\item[(2)]
the set $Z(\R)\cap\ol S$ is either empty or Zariski dense in $Z$.
\end{itemize}
(Here, of course, $\ol S$ denotes the closure of $S$ in $X(\R)$.)
\end{dfn}

Note that (1) is saying that $Z$ has codimension one in $X$ and is
not contained in the singular locus of $X$. Condition (2) says (for
normal $X$) that every irreducible component of $X\setminus V$ is
either compatible with $S$ or disjoint from $\ol S$ (c.f.\ Definition
\ref{dfncomp}).

\begin{examples}\label{excompl}
\hfil

1.\
Given a \sa\ subset $S$ of $\R^n$, the natural completion $\A^n
\subset\P^n$ of affine $n$-space is compatible with $S$ if and only
if $S$ contains an open cone in $\R^n$ (not necessarily centered at
the origin).

2.\
The case of curves is simple: Given a (possibly singular) irreducible
curve $C$, there is a unique projective completion $C\into X$ for
which $X_\sing\subset C_\sing$ (see \cite{Sch:MZ} 4.6). The
completion $X$ is compatible with any \sa\ subset $S$ of $C(\R)$. The
points of $X\setminus C$ are called the \emph{points of $C$ at
infinity}.

3.\
There are interesting classes of \sa\ sets $S$ (in $\R^n$, say) for
which $S$-compatible completions (of $V=\A^n$, in this case) can be
constructed as toric varieties. For example, when $S$ is defined by
(finitely many) binomial inequalities $ax^\alpha<bx^\beta$, this
always is the case. For such $S$, the ring $B(S)$ can be identified
explicitly in terms of the defining inequalities, and $B(S)$ is
always finitely generated as an $\R$-algebra. We plan to
expand on this remark in a future publication.
\end{examples}

Our interest in compatible completions arises from the next result.
It shows that such a completion calculates the ring of bounded
polynomials on a \sa\ set:

\begin{thm}\label{gococalxbs}
Let $V$ be an affine normal $\R$-variety, let $S\subset V(\R)$ be a
\sa\ subset, and assume that the completion $V\into X$ of $V$ is
compatible with $S$. Let $Y$ denote the union of those irreducible
components $Z$ of $X\setminus V$ for which $\ol S\cap Z(\R)=
\emptyset$, and put $U=X\setminus Y$. Then the inclusion $V\subset U$
induces a ring isomorphism
$$\scrO_X(U)\>\isoto\>B_V(S).$$
\end{thm}

\begin{proof}
Again, $\ol S$ denotes the closure of $S$ in $X(\R)$. Since $\ol S$
is compact and contained in $U(\R)$, every element of $\scrO_X(U)$ is
bounded on $S$. So the image of the restriction map $\scrO_X(U)\into
\R[V]$ is contained in $B_V(S)$. It remains to show that every $f\in
B_V(S)$, considered as a rational function on $U$, is regular on $U$.
Since $U$ is normal, it suffices that $v_Z(f)\ge0$ for every
irreducible component $Z$ of $U\setminus V$. By the construction of
$U$, and since $X$ is compatible with $S$, the intersection $\ol S
\cap Z(\R)$ is Zariski dense in $Z$, which means that the divisor $Z$
is compatible with the set $S$ (see \ref{dfncomp}). Since $f$ is
regular on $V$ and bounded on $S$, Lemma \ref{bddfctcompat} implies
$v_Z(f)\ge0$.
\end{proof}

\begin{lab}\label{exgocobs1}
We illustrate the possible use of Thm.\ \ref{gococalxbs} by two
examples. We use homogeneous coordinates $(x_0:x_1:x_2)$ on $\P^2$,
and we identify $(x,y)\in\A^2$ with $(1:x:y)\in\P^2$. Let $L=\{x_0=
0\}$ be the line at infinity in $\P^2$.

Consider the strip $S=\{(x,y)\colon|x|\le1\}$ in $\R^2$. In $\P^2
(\R)$ we have $\ol S\cap L(\R)=\{P\}$ where $P=(0:0:1)$. To improve
this we consider the blowing-up $\pi_1\colon X_1\to\P^2$ at $P$. Then
$\A^2\into X_1$ is a completion of $\A^2$ for which $X_1\setminus
\A^2$ has two irreducible components, namely $L'$ (the strict
transform of $L$) and $E_1=\pi_1^{-1}(P)$. In $X_1(\R)$ we have
$\ol S\cap L'(\R)=\emptyset$, and $\ol S\cap E_1(\R)$ is Zariski
dense in $E_1$. Therefore the completion $X_1$ is compatible with
$S$, and by \ref{gococalxbs} we have $B(S)=\scrO(U)$ for $U:=X_1
\setminus L'$. Clearly $\scrO(U)=\{f\in\R[x,y]\colon v_{E_1}(f)
\ge0\}$. Calculating the valuation $v_{E_1}$ we find
$$v_{E_1}\Bigl(\sum_{i,j}a_{ij}x^iy^j\Bigr)\>=\>\min\{-j\colon a_{ij}
\ne0\}.$$
So we get $B(S)=\scrO(U)=\R[x]$.
\end{lab}

\begin{lab}\label{exgocobs2}
For another example let now $T=\{(x,y)\in\R^2\colon|x|\le1$, $|xy|\le
1\}$. As before we use the blowing-up $\pi_1\colon X_1\to\P^2$ of the
plane in $P=(0:0:1)$. Again we have $\ol T\cap L'(\R)=\emptyset$ in
$X_1(\R)$, but this time $\ol T\cap E_1(\R)=\{P_1\}$ is a singleton.
Therefore we blow up $X_1$ in $P_1$ to get $\pi_2\colon X_2\to X_1$,
with exceptional fibre $E_2=\pi_2^{-1}(P_1)$. Then in $X_2(\R)$ we
find $\ol T\cap E'_1(\R)=\emptyset$, and $\ol T\cap E_2(\R)$ is
Zariski dense in $E_2$. So $X_2$ is a completion of $\A^2$ which is
compatible with $T$, and $B(T)=\scrO(W)$ for $W:=X_2\setminus
(L''\cup E_1')$ according to Thm.\ \ref{gococalxbs}. We find $\scrO
(W)=\{f\in\R[x,y]\colon v_{E_2}(f)\ge0\}$ and
$$v_{E_2}\Bigl(\sum_{i,j}a_{ij}x^iy^j\Bigr)\>=\>\min\{i-j\colon
a_{ij}\ne0\},$$
which shows $B(T)=\scrO(W)=\R[x,xy]$.
\end{lab}

In general, if we want to apply Theorem \ref{gococalxbs} to
calculate $B_V(S)$, we can start with some (normal) completion
$V\into X_0$ of $V$. By making suitable iterated blowing-ups with
centers over $X_0\setminus V$, we try to ``straighten out'' the set
$S$ at infinity more and more. When $V$ is a surface and the set $S$
is sufficiently regular at infinity, this procedure will always
lead, after finitely many steps, to a completion of $V$ which is
compatible with $S$, see Theorem \ref{exgocoreginf} below.


\section{Existence of compatible completions}\label{exgoco}

In view of Theorem \ref{gococalxbs}, we are now discussing the
existence question for $S$-compatible completions. We start with the
case $S=V(\R)$.

\begin{thm}\label{exgocovr}
Every connected non-singular $\R$-variety $V$ has a completion $X$
which is compatible with the set $S=V(\R)$, and such that the
irreducible components of $X\setminus V$ are non-singular. If $V$ is
quasi-projective then $X$ can be chosen to be projective.
\end{thm}

\begin{proof}
Start with any open dense embedding $V\into X$ into a complete
$\R$-variety $X$ (which we can take projective if $V$ is
quasi-projective). The singularities of $X$ are contained in
$X\setminus V$, and by resolving them we get $X$ non-singular.
Some irreducible components of $X\setminus V$ may
have codimension $\ge2$. This can be remedied by blowing up $X$ in
these centers.
So we can assume that every irreducible component of $X\setminus V$
has codimension one in $X$. Finally, by
embedded resolution of singularities, we can achieve that the
irreducible components of $X\setminus V$ are non-singular.

We claim that the completion $V\subset X$ is compatible with $V(\R)$.
To see this, fix an irreducible component $Z$ of $X\setminus V$ for
which $Z(\R)\ne\emptyset$. Since $Z$ is non-singular, the function
field $\R(Z)$ is real. By the Baer-Krull theorem, $\R(X)$ has an
ordering $\alpha$ which is compatible with the discrete valuation
$v_Z$. In particular, $Z$ is compatible with the set $V(\R)$
(Definition \ref{dfncomp}).
\end{proof}

We now turn to compatible completions for more general \sa\ subsets
$S$ of $V(\R)$. We shall denote the interior of a set $M\subset
V(\R)$ by $\inter(M)$.

\begin{dfn}\label{dfnreginfty}
Let $V$ be an irreducible $\R$-variety, let $S\subset V(\R)$ be a \sa\
set.
\begin{itemize}
\item[(a)]
The set $S$ is said to be \emph{regular} if $S\subset\ol{\inter\bigl(
S\cap V_\reg(\R)\bigr)}$.
\item[(b)]
$S$ is called \emph{regular at infinity} if $S=S_0\cup S_1$ where
$S_0$, $S_1$ are \sa\ sets with $\ol S_0$ compact and $S_1$ regular.
\item[(c)]
$S$ is called \emph{Zariski dense at infinity} if $S\setminus(K\cap
S)$ is Zariski dense in $V$ for every compact subset $K$ of $V(\R)$.
\end{itemize}
\end{dfn}

\begin{rem}
If $S$ is regular at infinity and $\ol S$ is not compact, then $S$ is
Zariski dense at infinity. Indeed, otherwise there would be a proper
Zariski closed subset $Z$ of $V$ and a compact set $K\subset V(\R)$
with $S\subset K\cup Z(\R)$. This would imply $\inter(S\cap V_\reg
(\R))\subset K$, and by regularity at infinity we would get $\ol S$
compact, a contradiction.
\end{rem}

We will need the following version of embedded resolution of
singularities on a surface. By a curve on $X$ we mean a reduced
effective divisor on $X$.

\begin{thm}\label{relembres}
Let $k$ be an infinite field, $X$ a normal quasi-projective surface
over $k$, $C$ a curve in $X$ and $T$ a finite set of closed points in
$C_\sing\cap X_\reg$. Then there exists a birational morphism
$\varphi\colon\wt X\to X$ of $k$-varieties with the following
properties:
\begin{enumerate}
\item
$\varphi$ induces an isomorphism $\wt X\setminus\varphi^{-1}(T)\isoto
X\setminus T$;
\item
$\wt X$ is quasi-projective and $\wt X_\sing=\varphi^{-1}(X_\sing)$
(in particular, $\wt X$ is normal);
\item
in all points of $\varphi^{-1}(T)$, the divisor $\varphi^{-1}(C)$ on
$\wt X$ has normal crossings and non-singular components.
\end{enumerate}
\end{thm}

\begin{proof}
If $X$ is non-singular and $T=C_\sing$, our statement becomes the
usual one for embedded resolution of curves in surfaces (see for
example \cite{Ha} Thm.\ V.3.7, or \cite{Cu} Sect.\ 3.5 for the case
of an arbitrary infinite base field). This implies the above version,
since $T$ is contained in $X_\reg$. In more detail, let $S=X_\sing
\cup(C_\sing\setminus T)$ and put $X_0=X\setminus S$, a non-singular
quasi-projective surface. Applying usual embedded resolution of
singularities to the divisor $C\cap X_0$ on $X_0$ and glueing it with
the identity of $X\setminus T$ yields the above version.
\end{proof}

The following result proves the existence of compatible completions
for surfaces, when the \sa\ set is regular at infinity (Definition
\ref{dfnreginfty}):

\begin{thm}\label{exgocoreginf}
Let $V$ be a connected normal quasi-projective surface over $\R$, and
let $S$ be a \sa\ subset of $V(\R)$ that is regular at infinity. Then
$V$ has a projective completion which is compatible with $S$. If $V$
is non-singular then the completion can be chosen to be non-singular
as well.
\end{thm}

\begin{proof}
We may assume that the function field $\R(V)$ is real. For otherwise,
$V(\R)$ consists of finitely many singular points of $V$. In that
case, any normal completion of $V$ will be compatible with any subset
of $V(\R)$.

We can also assume that $S$ is closed in $V(\R)$. Start with any
completion $V\into X_1$ of $V$. Since $V$ is a normal surface, the
singular set $V_\sing$ is finite. By resolving singularities of $X_1
\setminus V_\sing$ and gluing the resulting surface to $V$ we can
assume $X_{1,\sing}\subset V$. Let $\ol S$ be the closure of $S$ in
$X_1(\R)$, let $\partial\ol S$ be its boundary in $X_1(\R)$, and
denote by $C_1$ the Zariski closure of $\partial\ol S$ in $X_1$, a
curve on $X_1$.
Further write $D_1=X_1\setminus V$. Now apply embedded resolution, as
stated in Theorem \ref{relembres}, to the surface $X_1$, the curve
$C_1\cup D_1$ on $X$ and the set $T:=(C_1\cup D_1)_\sing\cap D_1$. We
obtain a birational morphism $\pi\colon X\to X_1$ whose restriction
$\pi^{-1}(V)\to V$ is an isomorphism. Via $\pi$ we can regard $X$ as
a completion of $V$. Let $D=X\setminus V=\pi^{-1}(D_1)$ and $C=
\pi^{-1}(C)$. The completion $V\into X$ has the following properties:
\begin{enumerate}
\item
$X$ is projective and normal (non-singular if $V$ is non-singular);
\item
the irreducible components of the divisor $C\cup D$ are non-singular
in all points of $D$, and if two components of $C\cup D$ meet in a
point of $D$, then that point is a normal crossing singularity.
\end{enumerate}
We show that $V\into X$ is compatible with $S$.
Let $Z$ be an irreducible component of $D=X\setminus V$ for which the
set $\ol S\cap Z(\R)$ is non-empty. We have to show that this set is
Zariski dense in $Z$. Let $x\in\ol S\cap Z(\R)$, and let $u\in\scrO_
{X,x}$ be a local equation for $Z$. There can be at most one more
irreducible component $Z'\ne Z$ of $C\cup D$ which passes through
$x$, by property (2). If such $Z'$ exists, let $v\in\scrO_{X,x}$ be a
local equation for $Z'$. Then $u,\,v$ is a regular system of
parameters for $\scrO_{X,x}$. If no such $Z'$ exist we put $Z'=Z$ and
let $v\in\scrO_{X,x}$ be any element for which $u,\,v$ is a regular
system of parameters for $\scrO_{X,x}$.

Let $U$ be an open neighbourhood of $x$ in $X(\R)$ such that $U\cap
(C\cup D)(\R)\subset(Z\cup Z')(\R)$. By shrinking $U$ we can assume
that $u$ and $v$ are regular functions on $U$.
Shrinking $U$ further if necessary, there are precisely four
connected components of $\{y\in U\colon(uv)(y)\ne0\}$. Since $S$ is
regular at infinity, it follows that $\ol S$ contains at least one of
these four components. In particular, $\ol S\cap Z(\R)$ contains a
non-empty open subset of $Z(\R)$ and is therefore Zariski dense in
$Z$. This completes the proof.
\end{proof}

\begin{rems}
\hfil

1.\ We do not know whether Theorem \ref{exgocoreginf} extends to
$\dim(V)>2$. See however Theorem \ref{exgocovr}.  \smallskip

2.\
In Theorem \ref{exgocoreginf}, the regularity of $S$ at infinity is
not necessary in order that a compatible completion exists. This is
demonstrated by simple examples like the following: Let $V=\A^2$ and
consider the set
$$S\>=\>\{(x,y)\in\R^2\colon xy^2\ge0\}.$$
$S$ is the union of the right half plane with the $x$-axis. Clearly
$\P^2$ is a completion of $\A^2$ which is compatible with $S$, even
though $S$ is not regular at infinity. On the other hand, when
$S\subset V(\R)$ is unbounded but contained in the union of a compact
set and a proper subvariety of $V$ (so $S$ fails to be Zariski dense
at infinity), there cannot be any completion of $V$ compatible with
$S$ (see Prop.~\ref{condconds}.)
\smallskip

3.\
If $C$ is a possibly singular (irreducible) affine curve, it is still
true that $C$ has a unique projective completion $X$ which is
non-singular in the added points (c.f.\ also \cite{Sch:MZ} Lemma 4.5).
The completion $X$ is compatible with every \sa\ subset $S$ of
$C(\R)$, and Theorem \ref{gococalxbs} holds in this case as well.
\smallskip

4.\ Let $X$ be a non-singular connected projective surface over $\R$,
and let $D$ be a curve on $X$ with irreducible components $Z_1,\dots,
Z_r$. Then it is sometimes possible to calculate the transcendence
degree of the ring $B=\scrO_X(X\setminus D)$ from the intersection
matrix $M_D=(Z_i\,.\,Z_j)_{i,j=1,\dots,r}$ of the divisor $D$:
\begin{enumerate}
\item
If $M_D$ is negative definite, then $\trdeg_\R(B)=0$, hence $B=\R$.
\item
If $M_D$ has a positive eigenvalue, then $\trdeg_\R(B)=2$.
\end{enumerate}
If $M_D$ is negative semidefinite and singular, there is no general
statement about the transcendence degree of $B$. (See \cite{Ii} 8.3
for proofs.)

For instance, consider $T=\{(x,y)\in\R^2\colon|x|\le1$, $|xy|\le1\}$
as in Example \ref{exgocobs2}. There we found $B(T)=\scrO(X_2
\setminus D)$ with $D=L''\cup E_1'$ (using the notation from
\ref{exgocobs2}). The intersection matrix of $D$ is
$$M_D=\left(\begin{array}{rr}0&1\\1&-2\end{array}\right),$$
which has a positive eigenvalue. Thus we conclude $\trdeg\,B(T)=2$
without actually calculating the ring $B(T)$ as in \ref{exgocobs2}.
In this example, the transcendence degree can also be read off using
Thm.~\ref{4conds}, since the birational map $\A^2\to\A^2$, $(x,y)
\mapsto(x,xy)$ is bounded on $S$.
\end{rems}

Let $V$ be an affine normal $\R$-variety and let $S\subset V(\R)$ be
a \sa\ subset. We have seen that the ring $B_V(S)$ is isomorphic to a
ring $\scrO_U(U)$ for some quasi-projective $\R$-variety $U$, provided
that $V$ has a completion which is compatible with $S$. We are now
going to prove a converse (Thm.\ \ref{owalsbs} below). It roughly
states, for every quasiprojective $U$ which is real, that the ring
$\scrO_U(U)$ can be realized as $B_V(S)$ for some open affine subset
$V$ of $U$ and some \sa\ subset $S$ of $V(\R)$.

\begin{lem}\label{compnotouch}
Let $V$ be an irreducible affine $\R$-variety, and let $X$ be a
projective completion of $V$ which is normal. Let $Y_1,\dots,Y_r$ be
irreducible components of $X\setminus V$ which are real, and let
$Z_1,\dots,Z_s$ be the remaining components (real or not) of
$X\setminus V$. There exists a regular and basic closed \sa\ subset
$S$ of $V(\R)$ that is compatible with $Y_1,\dots,Y_r$ and whose
closure in $X(\R)$ is disjoint from $Z_1(\R),\dots,Z_s(\R)$. In
particular, the completion $X$ of $V$ is compatible with $S$.
\end{lem}

\begin{proof}
Fix an index $i\in\{1,\dots,r\}$. Since $X$ is normal and $Y_i$ is
real, we find a point $p_i\in Y_i(\R)$ which is non-singular on $X$
and on $X\setminus V$. By this last fact, $p_i$ has a generalisation
$\beta_i$ in $\Sper\R(Y_i)$. Since $X$ is normal, $\beta_i$ has a
generalization $\alpha_i$ in $\Sper\R(X)$.
Every \sa\ subset $S$ of $X(\R)$ with $\alpha_i\in\wt S$ is
compatible with $Y_i$. There is an open affine subset $U$ of $X$ such
that $U(\R)=X(\R)$.
Let $x_1,\dots,x_n$ be a system of generators for the $\R$-algebra
$\R[U]$. If $c_i\in\R$ is positive and sufficiently small, the closed
\sa\ set
$$S_i:=\Bigl\{p\in V(\R)\colon \sum_{k=1}^n(x_k(p)-x_k(p_i))^2\le c_i
\Bigr\}$$
will be contained in $V_\reg(\R)$ and regular, and moreover $S_i$
will be compatible with $Y_i$ and disjoint from $Z_1(\R)\cup\cdots
\cup Z_s(\R)$. This being done for $i=1,\dots,r$, we may further
assume that $S_1,\dots,S_r$ are pairwise disjoint, by making $c_1,
\dots,c_r$ even smaller if necessary.

We claim that $S:=S_1\cup\cdots\cup S_r$ is a basic closed \sa\
subset of $V(\R)$, which will complete the proof.
For $i=1,\dots,r$ put
$$g_i:=c_i-\sum_{k=1}^n(x_k-x_k(p_i))^2\;\in\R[U],$$
and let $V_0$ be the maximal Zariski open subset of $V$ for which
$g_i|_{V_0}\in\scrO_V(V_0)$ for $i=1,\dots,r$. Since $V$ is normal,
the closed subset $Z=V\setminus V_0$ has pure codimension one in $V$.
Clearly $Z(\R)=\emptyset$ since $U(\R)=X(\R)$ and the $g_i$ are
regular on $U$. Choose $h_1,\dots,h_k\in\R[V]$ which generate the
vanishing ideal of $Z$ in $\R[V]$, and put $h:=h_1^2+\cdots+h_k^2$.
Then $h$ has no real zeros on $V$, and since $V$ is normal, there
exists $N\ge1$ such that $f_i:=h^Ng_i\in\R[V]$ for $i=1,\dots,r$.
Thus $S_i=\{p\in V(\R)\colon f_i(p)\ge0\}$ for every $i$, and hence
$$S_1\cup\cdots\cup S_r\>=\>\bigl\{p\in V(\R)\colon(-1)^{r+1}f_1(p)
\cdots f_r(p)\ge0\bigr\},$$
since the $S_i$ are pairwise disjoint.
\end{proof}

An effective Weil divisor $D$ on a normal $\R$-variety $X$ will be
called \emph{totally real} if every irreducible component of $D$ is
a real variety. By $\Cl(X)$ we denote the class group of Weil
divisors modulo linear equivalence on $X$. We will use the following
theorem:

\begin{thm}[Roggero \cite{Ro}]\label{roggero}
Let $V$ be a normal affine real $\R$-variety of dimension at least
$2$. Then every divisor on $V$ is linearly equivalent to a totally
real effective divisor on $V$.
\qed
\end{thm}

\begin{cor}\label{roggproj}
Let $X$ be a real normal projective $\R$-variety of dimension at
least~$2$, and let $H$ be an ample effective divisor on $X$. For
every divisor $D$ on $X$ there exists a totally real effective
divisor $D_0$ on $X$ such that $D\sim D_0+H_0$, where $H_0$ is a
linear combination of irreducible components of $H$.
\end{cor}

\begin{proof}
Let $V=X\setminus H$. The kernel of the restriction map $\Cl(X)\to\Cl
(V)$ is generated by (the class of) $H$. Thus it suffices to apply
Roggero's theorem to the affine variety $V$.
\end{proof}

\begin{lem}\label{realample}
Let $U$ be a connected normal quasi-projective $\R$-variety which is
real. There exist real prime divisors $Y_1,\dots,Y_r$ on $U$ such
that the variety $U\setminus(Y_1\cup\cdots\cup Y_r)$ is affine.
\end{lem}

\begin{proof}
If $U$ is a curve, we can take $r=1$ and $Y_1$ any real point on
$U$. So we may assume $\dim(U)\ge 2$.  Assume furthermore that $U$
is projective. Fix a non-singular point $p\in U(\R)$ and an ample
divisor $H$ on $U$, and let $L\subset |H|$ consist of those $D\in
|H|$ that pass through $p$. The linear system $L$ has no fixed
components, so by Bertini's theorem (see for example Jouanolou
\cite{Jo}, Cor.~6.11), the set of all $Z\in L$ for which $Z$ is
irreducible and satisfies $Z_\sing\subset U_\sing$ has non-empty
interior inside $L$. In particular, it is non-empty. Any such $Z$ is
real since it contains $p$ as a regular real point. Moreover
$U\setminus Z$ is affine since $Z$ is ample.

When $U$ is only quasi-projective, fix an open dense embedding
$U\into X$ with $X$ normal and projective and such that $X\setminus
U$ has pure codimension one in $X$ (the latter can be achieved by
blowing up if necessary). Let $E$ be the divisor $X\setminus U$. By
the projective case above we find an ample prime divisor $H$ on
$X$ that is totally real. By Cor.\ \ref{roggproj}, there exist $n\in
\Z$ and a totally real effective divisor $D$ on $X$ such that
$-E\sim D+nH$.
For $m>\max\{0,n\}$ the divisor $E+D+mH$ is effective and ample. It
follows that the irreducible components $Y_1,\dots,Y_r$ of $(D+mH)
\cap U$ are real and $U\setminus(Y_1\cup\cdots\cup Y_r)$ is affine,
as desired.
\end{proof}

\begin{thm}\label{owalsbs}
Let $U$ be a connected normal quasi-projective $\R$-variety which
is real. There exists an open affine subset $V$ of $U$ and a regular
and basic closed subset $S$ of $V(\R)$ such that $\scrO_U(U)\isoto
B_V(S)$.
\end{thm}

\begin{proof}
Let $U\into X$ be a normal projective completion of $U$, and let
$Z_1,\dots,Z_s$ be the irreducible components of $X\setminus U$. By
Lemma \ref{realample} there exist further prime divisors $Y_1,\dots,
Y_r$ on $X$ which are real and such that the open subset $V:=
X\setminus\bigl(\bigcup_iY_i\cup\bigcup_jZ_j\bigr)$ of $U$ is affine.
By Lemma \ref{compnotouch} we find a regular and basic closed \sa\
set $S$ in $V(\R)$ which is compatible with $Y_1,\dots,Y_r$ and which
satisfies $S\cap Z_j(\R)=\emptyset$ for $j=1,\dots,s$. In particular,
the completion $X$ of $V$ is compatible with $S$. Therefore $\scrO_U
(U)\isoto B_V(S)$ by Theorem \ref{gococalxbs}.
\end{proof}

Under certain conditions we know that the $\R$-algebra $B_V(S)$ is
finitely generated, see e.g.\ Thm.\ \ref{bvsfg} below. Writing $W:=
\Spec B_V(S)$ for the affine $\R$-variety defined by this ring, the
canonical morphism $\varphi_S\colon V\to W$ is universal for 
morphisms that are bounded on $S$ from $V$ to affine varieties, 
i.e., every morphism $V\to W'$ into an affine $\R$-variety $W'$ which 
is bounded on $S$ factors uniquely through $\varphi_S$. Note 
that the variety $W$ is normal, provided $V$ is normal (Lemma 
\ref{soriten}(e)).  
We are wondering how to characterize all morphisms $V\to W$ of normal 
affine $\R$-varieties which are of the form $\varphi_S\colon V\to 
\Spec B_V(S)$ for some \sa\ subset $S$ of $V(\R)$ (with $B_V(S)$ 
finitely generated).

Although we do not know a general answer to this question, we can
present the following sufficient condition for birational morphisms:

\begin{prop}\label{oualsbvs}
Let $\varphi\colon V\to W$ be a birational morphism of connected
normal real affine $\R$-varieties. Assume there is a totally real
effective divisor $Z$ on $W$ such that $\varphi$ restricts to an
isomorphism $\varphi^{-1}(W\setminus Z)\isoto W\setminus Z$. Then
there exists a
\sa\ set $S\subset V(\R)$ such that $\varphi^*\R[W]=B_V(S)$.
\end{prop}

For the proof we need the following lemma:

\begin{lem}\label{prescrcomp}
Let $V$ be an affine normal $\R$-variety, and let $Z$ be a given
effective divisor on $V$ which is totally real. There exists a
compact
\sa\ set $S$ in $V(\R)$ which is compatible with every irreducible
component of $Z$, and such that $S\setminus(S\cap Z(\R))$ is dense
in $S$.
\end{lem}

\begin{proof}
Let $Z_1,\dots,Z_r$ be the irreducible components of $Z$, and fix
$i\in\{1,\dots,r\}$.
Since $Z_i$ is real and not contained in $V_\sing$, we find a point
$p_i\in Z_i(\R)$ which is nonsingular on $Z_i$ and on $V$. So $p_i$
has a generalization $\beta_i$ in $\Sper\R(Z_i)$, and in turn,
$\beta_i$ has a generalization $\alpha_i$ in $\Sper\R(V)$. Every \sa\
subset $S$ of $V(\R)$ with $\alpha_i\in\wt S$ is compatible with
$Z_i$. If $x_1,\dots,x_n$ is a system of generators of the
$\R$-algebra $\R[V]$, we may therefore define $S$ to be the closure
of the set $\{p\in V(\R)\colon p\notin Z(\R)$, $x_1(p)^2+\cdots+
x_n(p)^2<c\}$, where $c\in\R$ is positive and sufficiently large.
\end{proof}

\begin{proof}[Proof of Prop.\ \ref{oualsbvs}]
By Lemma \ref{prescrcomp} there exists a compact basic closed \sa\
subset $T$ of $W(\R)$ which is compatible with every irreducible
component of the divisor $Z$, and such that $T\setminus(T\cap Z(\R))$
is dense in $T$. Let $S:=\varphi^{-1}(T)\subset V(\R)$.
Since $\varphi(S)$ contains $T\setminus(T\cap Z(\R))$ by the
hypothesis, we have $\ol{\varphi(S)}=T$. By compactness of $T$ it is
clear that $\varphi^*(\R[W])\subset B_V(S)$.
We claim that this inclusion is an equality. Let $f\in\R[V]$ with
$f\notin\varphi^*(\R[W])$. Considering $f$ as a rational function on
$W$, it follows that $f$ has a pole along one of the irreducible
components of $Z$, since $W$ is normal. Since $T$ is compatible with
that component, it follows from Lemma \ref{bddfctcompat} that $f$ is
unbounded on $T\cap\dom_W(f)$. In particular, $f$ is unbounded on
$T\setminus(T\cap Z(\R))$,
and so $f$ is unbounded on $S$.
\end{proof}


\section{Finite generation of the ring of bounded polynomials}

We first recall that $B_C(S)$ is always finitely generated when $C$
ist a curve:

\begin{prop}
Let $C$ be an irreducible affine curve over $\R$, possibly singular,
and let $S\subset C(\R)$ be a \sa\ subset. Assume that $C$ is real.
\begin{itemize}
\item[(a)]
$B_C(S)$ is finitely generated as an $\R$-algebra.
\item[(b)]
If every point of $C$ at infinity is real and lies in the closure of
$S$, then $B_C(S)=\R$.
\item[(c)]
Otherwise $B_C(S)$ has transcendence degree one over $\R$.
\end{itemize}
\end{prop}

\begin{proof}
Let $C\into X$ be the completion of $C$ for which $X_\sing\subset
C_\sing$ (\ref{excompl}.2), and let $C'\subset X$ be the open
subset which is the union of $C$ and all $\R$-points of $X\setminus
C$ which lie in the closure of $S$. Then $\scrO_X(C')\isoto B_C(S)$.
Indeed, even though Thm.\ \ref{gococalxbs} is not directly applicable
here when $C$ is singular, an inspection of the proof shows that it
applies nevertheless. When $C'\ne X$ then the curve $C'$ is affine,
and so $B_C(S)$ is a finitely generated $\R$-algebra of transcendence
degree one. When $C'=X$ (which means that the hypothesis of (b) is
satisfied) then $B_C(S)=\R$.
\end{proof}

Before proceeding to our main result on finite generation, we study
some consequences for the ring $B_V(S)$ that arise from the existence
of an $S$-compatible completion.

\begin{lemdfn}\label{zarclosinf}
Let $V$ be an irreducible $\R$-variety, and let $S$ be a \sa\ subset
of $V(\R)$. There is a unique smallest Zariski closed subset $Z$ of
$V$ such that $S\subset Z(\R)\cup K$ for some compact subset $K$ of
$V(\R)$. We shall call $Z$ the \emph{Zariski closure of $S$ at
infinity}.
\end{lemdfn}

\begin{proof}
If Zariski closed subsets $Z_1$, $Z_2$ of $V$ and compact subsets
$K_1$, $K_2$ of $V(\R)$ are given with $S\subset Z_i(\R)\cup K_i$
($i=1,2$), then $S$ is contained in $\bigl(Z_1(\R)\cap Z_2(\R)\bigr)
\cup(K_1\cup K_2)$. This implies the assertion.
\end{proof}

Clearly, $S$ is Zariski dense at infinity (Definition
\ref{dfnreginfty}) if and only if its Zariski closure at infinity is
$V$.

\begin{prop}\label{zarclinfcond}
Let $V$ be an irreducible affine $\R$-variety and $S\subset V(\R)$ a
\sa\ set. Let $Z$ be the Zariski closure of $S$ at infinity, and let
$I_Z$ be the full vanishing ideal of $Z$ in $\R[V]$. Then $I_Z$ is
equal to the conductor of $B_V(S)$ in $\R[V]$, that is, $I_Z$ is the
largest ideal of $\R[V]$ which is contained in $B_V(S)$.
\end{prop}

\begin{proof}
Since $S\subset Z(\R)\cup K$ for some compact set $K$, it is clear
that $I_Z\subset B_V(S)$. Conversely let $b\in B_V(S)$ with $b\notin
I_Z$.
The subset $\{x\in S\colon b(x)\ne0\}$ of $V(\R)$ is unbounded. For
if its closure $K$ were compact, we would have $S\subset K\cup
\{b=0\}$,
contradicting the definition of $Z$. Hence there exists a
\sa\ curve $\gamma\colon\left[0,\infty\right[\to S$, $t\mapsto
\gamma_t$ such that $b(\gamma_t)\ne0$ for all $t\ge0$ and such that
$\Gamma:=\{\gamma_t\colon t\ge0\}$ is unbounded in $V(\R)$. This last
condition means that there exists $f\in\R[V]$ for which
$f(\gamma_t)$, $t\ge0$, is unbounded.
Since $b\ne0$ on $\Gamma$, it follows from Lemma \ref{unbddcurvegerm}
below that the product $f^nb$ is unbounded on $\Gamma$ for
sufficiently large $n\ge1$. So $bf^n\notin B_V(S)$.
\end{proof}

The following well-known fact was used in the last proof:

\begin{lem}\label{unbddcurvegerm}
Let $f$, $g\colon\left[0,\infty\right[\to\R$ be two continuous \sa\
functions such that $f$ is unbounded and $g>0$ everywhere. Then
$f^ng$ is unbounded for sufficiently large $n\ge1$.
\qed
\end{lem}

\begin{lem}\label{conduct}
Let $B\subset A$ be a ring extension, and let $J$ be an ideal of $A$
which is contained in $B$. If $J$ contains an element which is not a
zero divisor of $A$, and if $J$ is finitely generated as an ideal of
$B$, then the extension $B\subset A$ is integral.
\end{lem}

\begin{proof}
This is a special case of \cite[Cor.~4.6]{Ei}.
\end{proof}

\begin{lem}\label{condzero}
Let $Y$ be a connected normal $\R$-variety, and let $V\subset Y$ be
an affine open subset. If $V\ne Y$ then $\scrO_Y(Y)$ contains no
non-zero ideal of $\R[V]$.
\end{lem}

\begin{proof}
The inclusion $\scrO(Y)\subset\R[V]$ is proper. Indeed, otherwise
there would be a morphism $Y\to V$ which is the identity on $V$,
which is absurd if $V\ne Y$.
Assume $b\in\scrO(Y)$ is such that $bf\in\scrO(Y)$ for every $f\in
\R[V]$. Choose $f\in\R[V]$ with $f\notin\scrO(Y)$. Since $Y$ is
normal this means that there exists a prime divisor $Z$ of $Y$ with
$v_Z(f)\le-1$. Unless $b=0$ we therefore get $v_Z(bf^n)\le-1$ for
large $n>0$, and so $bf^n\notin\scrO(Y)$. This proves the assertion.
\end{proof}

\begin{prop}\label{condconds}
Let $V$ be a connected normal affine $\R$-variety, and let $S$ be an
unbounded \sa\ subset of $V(\R)$. Consider the following conditions:
\begin{itemize}
\item[(i)]
$V$ has a completion which is compatible with $S$;
\item[(ii)]
the ring $B_V(S)$ is noetherian;
\item[(iii)]
the conductor of $B_V(S)$ in $\R[V]$ is zero;
\item[(iv)]
$S$ is Zariski dense in $V$ at infinity.
\end{itemize}
We have (i) $\To$ (iii), (ii) $\To$ (iii) and (iii) $\iff$ (iv).
\end{prop}

\begin{proof}
(i) $\To$ (iii):
Let $V\into X$ be a completion which is compatible with $S$, let $Y$
be the union of the irreducible components $Z$ of $X\setminus V$
with $Z(\R)=\emptyset$, and let $U=X\setminus Y$. By Thm.\
\ref{gococalxbs}, the inclusion $V\into U$ induces an isomorphism
$\scrO(U)\isoto B_V(S)$. Since $S$ is unbounded we have $U\ne V$. By
Lemma \ref{condzero}, $B_V(S)$ contains no non-zero ideal of $\R[V]$.

(ii) $\To$ (iii):
We have $B\ne\R[V]$ since $S$ is unbounded. Let $J$ be an ideal of
$\R[V]$ which is contained in $B_V(S)$. Since $B$ is noetherian,
Lemma \ref{conduct} implies $J=(0)$.

The equivalence of (iii) and (iv) follows from Prop.\
\ref{zarclinfcond}.
\end{proof}

\begin{cor}\label{thinnotfg}
Let $V$ be a connected normal affine $\R$-variety, and let $S$ be an
unbounded \sa\ subset of $V(\R)$. If $S$ is not Zariski dense at
infinity, then $B_V(S)$ is not noetherian.
\end{cor}

\begin{proof}
The vanishing ideal of the Zariski closure of $S$ at infinity is a
proper non-zero ideal of $\R[V]$ that is contained in $B_V(S)$. Hence
the conductor of $B_V(S)$ in $\R[V]$ is non-zero, and so $B_V(S)$ is
not noetherian by Prop.\ \ref{condconds}.
\end{proof}

\begin{rem}
Condition (i) in Prop.\ \ref{condconds} implies $B_V(S)\cong\scrO_U
(U)$ for some $\R$-variety $U$ (Thm.\ \ref{gococalxbs}). In general,
the ring $\scrO_U (U)$, and hence also $B_V(S)$, may fail to be
noetherian, see \ref{vakilex} below. Therefore, (i) does not imply
(ii) in Prop.\ \ref{condconds}.
\end{rem}

The implication (ii) $\To$ (iv) in Prop.\ \ref{condconds} can be
strengthened as follows:

\begin{prop}\label{niederdimteil}
Let $V$ be a normal affine $\R$-variety, and let $S$ be a \sa\ subset
of $V(\R)$ for which the ring $B_V(S)$ is noetherian. Then for every
\sa\ subset $T\subset S$ with $B_V(S)\ne B_V(T)$, the inequality
\begin{equation}\label{strangeq}
\trdeg_\R B_V(T)\>\le\>\dim(S\setminus T)
\end{equation}
holds.
\end{prop}

Assume that $S$ is unbounded and (for simplicity) closed, and that
$B_V(S)$ is noetherian. Then conclusion \eqref{strangeq} implies for
any compact subset $T$ of $S$ that $S\setminus T$ is Zariski dense in
$V(\R)$. This means that $S$ is Zariski dense at infinity, and shows
that the implication  (ii) $\To$ (iv) of \ref{condconds} is contained
in Prop.\ \ref{niederdimteil}.

\begin{proof}
Let $Z$ be the Zariski closure of $S\setminus T$ in $V$, and let $I=
I_Z$ be the vanishing ideal of $Z$. Then $B(T)/I\cap B(T)\subset
\R[V]/I$, and so
$$\trdeg_\R B(T)/I\cap B(T)\>\le\>\trdeg_\R\R[V]/I\>=\>\dim Z\>=\>
\dim(S\setminus T).$$
If \eqref{strangeq} were false we would conclude $I\cap B(T)\ne(0)$.
Now the ideal $I\cap B(T)$ of $B(T)$ is contained in $B(S)$. This
contradicts Lemma \ref{conduct} since $B(S)$ is integrally closed in
$B(T)$ (even in $\R[V]$).
\end{proof}

We now consider algebraic surfaces. Theorem \ref{bvsfg} below is one
of our main results. It is based on the results of Sections
\ref{goco} and \ref{exgoco} and on the following theorem:

\begin{thm}[Zariski]\label{zar}
Let $W$ be a normal quasi-projective surface over a field $k$ of
characteristic zero. Then the ring $\scrO_W(W)$ is finitely generated
as a $k$-algebra.
\end{thm}

See Zariski \cite{Z}, remarks on Theorem~1 at the end of the article.
An alternative proof in modern language was given in \cite{Kr}, based
on ideas of Nagata.

\begin{thm}\label{bvsfg}
Let $V$ be a connected normal affine surface over $\R$. For every
\sa\ subset $S$ of $V(\R)$ which is regular at infinity, the
$\R$-algebra $B_V(S)$ is finitely generated.
\end{thm}

\begin{proof}
By Theorem \ref{exgocoreginf}, $V$ has a projective completion
$V\into X$ which is compatible with $S$. If $Y$ is the union of all
irreducible components of $X\setminus V$ without a real point, then
$\scrO_X(X\setminus Y)\isoto B_V(S)$ by Theorem \ref{gococalxbs}.
This implies the assertion, since $\scrO_X(X\setminus Y)$ is finitely
generated by Zariski's theorem \ref{zar}.
\end{proof}

\begin{lab}\label{vakilex}
It is folklore knowledge that Zariski's theorem \ref{zar} ceases to
be true in dimensions larger than two. However, it seems not so easy
to localize a reference for this fact in the published literature.

A construction of a quasi-projective non-singular threefold whose
ring of global sections is not finitely generated was given by Vakil
\cite{Va}: Take an elliptic curve $E$ over a field $k$, and let $L$,
$L'$ be invertible sheaves on $E$ such that $\deg(L')>0$, $\deg(L)=0$
and $L$ is not torsion in $\Pic(E)$. Let $W$ be the total space of
the vector bundle $L\oplus L'$ on $E$. Then the ring $\scrO_W(W)$ is
not noetherian. By a variation of this construction one can also
obtain another example which is even quasi-affine (see \cite{Va}).

Using this construction for $k=\R$ in a case where the elliptic curve
$E$ is real, we can apply Theorem \ref{owalsbs} and conclude:
\end{lab}

\begin{cor}\label{bvsnotft}
There exists a non-singular affine $\R$-variety $V$ of dimension
three and a regular and basic closed \sa\ subset $S$ of $V(\R)$ for
which the ring $B_V(S)$ is not noetherian.
\qed
\end{cor}

\begin{lab}
  From results of Kuroda \cite{Ku}, it is possible to conclude the
  existence of rational, normal, quasi-affine threefolds whose ring of
  regular functions is not finitely generated. This was pointed out to
  us by Sebastian Krug, answering a question in an earlier version of
  this paper. The details of the construction, which he communicated
  to us, are not presented here.
\end{lab}

\begin{lab}
By Zariski's theorem, an example as in \ref{vakilex} cannot exist in
dimension two. However, there are examples of (non-normal)
irreducible quasi-projective surfaces $W$ for which $\scrO_W(W)$ is
not finitely generated. One such construction of a quasi-affine
example is due to Nagata (\cite{Na1}, \cite{Na2}, see also \cite{Kr}
for a detailed account).

The construction of two-dimensional varieties whose ring of global
sections is not finitely generated becomes much easier when we allow
the variety to be reducible. Here is an example (again see \cite{Kr}
for details):
In affine $3$-space with coordinates $(x,y,z)$ let $U:=V(xy)$ (the
transversal union of two planes) and $L:=V(y,z)$ (a line in one of
the planes which intersects the other plane transversely). Then $W:=
U\setminus L$ is a reducible quasi-affine variety of dimension two
for which the ring $\scrO_W(W)$ is not noetherian.

To interpret this example in terms of bounded functions, let
$$V=U\setminus(U\cap V(z))=W\setminus V(x,z),$$
an open affine subvariety of $W$, and write $L'=V(x,z)$. Let $S_0$ be
a regular, compact and basic closed \sa\ subset of $U(\R)$ which is
disjoint from $L$ and compatible with $L'$. For example, $S_0=
\{(0,y,z)\colon(y-2)^2+z^2\le1\}$ will do. Then $S=S_0\cap V(\R)$ is
basic closed in $V(\R)$, and $B_V(S)=\scrO_W(W)$ is not noetherian.
\end{lab}


\end{document}